\documentclass[12 pt]{article}%
\usepackage{amsmath, amsfonts, amsthm, color,latexsym}
\usepackage{amsmath, ulem}
\usepackage{amsfonts}
\usepackage{amssymb}
\usepackage{color, soul}
\usepackage[all]{xy}
\usepackage{graphicx}%
\providecommand{\U}[1]{\protect\rule{.1in}{.1in}}
\allowdisplaybreaks[4]
\newtheorem{theorem}{Theorem}[section]
\newtheorem{proposition}[theorem]{Proposition}
\newtheorem{corollary}[theorem]{Corollary}
\newtheorem{example}[theorem]{Example}

\newtheorem{remark}[theorem]{Remark}

\newtheorem{lemma}[theorem]{Lemma}
\newtheorem{final remark}[theorem]{Final Remark}
\newtheorem{definition}[theorem]{Definition}
\allowdisplaybreaks[4]

\begin{document}

\title{Symmetric ideals of generalized summing multilinear operators}
\author{Geraldo Botelho\thanks{Supported by FAPEMIG Grant PPM-00450-17. } ~and  Ariel S. Santiago\thanks{Supported by a FAPEMIG scholarship\newline 2020 Mathematics Subject Classification: 46G25, 47H60, 47L22, 47B10.\newline Keywords: Banach spaces, multilinear operators, sequence classes, symmetric ideals. }}
\date{}
\maketitle

\begin{abstract} Let $X_1, \ldots, X_n,Y$ be classes of Banach spaces-valued sequences. An $n$-linear operator $A$ between Banach spaces belongs to the ideal of $(X_1, \ldots, X_n;Y)$-summing multilinear operators if $(A(x_j^1, \ldots, x_j^n))_{j=1}^\infty$ belongs to $Y$ whenever $(x_j^k)_{j= 1}^\infty$ belongs to $X_k, k = 1, \ldots, n$. In this paper we develop techniques to generate non trivial symmetric ideals of this type. Illustrative examples and additional applications of the techniques are provided.
\end{abstract}

\section{Introduction}
In the theory of ideals of linear and multilinear operators between Banach spaces, ideals defined, or characterized, by the transformation of vector-valued sequences play a central role. Since everything began with the ideal of absolutely summing linear operators, such ideals are called ideals of generalized summing operators. Symmetric ideals of multilinear operators were introduced in \cite{fg} and have been studied since then. An ideal $\cal M$ of multilinear operators is symmetric if $\cal M$ contains the symmetrizations of the operators belonging to $\cal M$. Very recent developments of symmetric ideals can be found in \cite{sergio, botelho+wood, baianos}. The main purpose of this paper is to develop techniques to create non trivial (in a sense that will become clear soon) symmetric ideals of generalized summing multilinear operators.

Let us recall one specific situation in which symmetric ideals are relevant. By $\widehat{A}$ we mean the homogeneous polynomial determined by the multilinear linear operator $A$, and by $\check{P}$ the symmetric multilinear operator associated to the homogeneous polynomial $P$. A given ideal of multilinear operators $\cal M$ gives rise to two polynomial ideals, namely,
$$ {\cal M}^\wedge = \{\widehat{A} : A \in {\cal M}\} \mbox{ and } {\cal M}^\vee = \{P : \check{P} \in {\cal M}\}.$$
It is often necessary that these two ideals coincide and, as is well known, $ {\cal M}^\wedge =  {\cal M}^\vee$ if and only if the ideal $\cal M$ of multilinear operators is symmetric.

Different approaches to ideals of generalized summing operators have appeared in the literature, see, e.g., \cite{nacib, bayart, davidson, dimant, maite, baianosmedit}. In this paper we follow the method from \cite{botelho+campos}, which is based on the concept of sequence classes. Roughly speaking, a sequence is a rule that, to each Banach space $E$, assigns a Banach space $X(E)$ of $E$-valued sequences enjoying certain properties. Precise definitions and examples will be given in Section 2. For the moment we say that $\ell_\infty(\cdot)$ denotes the sequence class $E \mapsto \ell_\infty(E)$. This approach to generalized summing operators has proved to be quite fruitful, see \cite{achour, achour2, botelho+campos+2, davidson, jamiljoed, baweja, baianosmedit, baianosmfat}. Given sequences classes $X_1, \ldots, X_n,Y$, an $n$-linear operator $A \colon E_1 \times \cdots \times E_n \longrightarrow F$ between Banach spaces is $(X_1, \ldots, X_n;Y)$-summing if $(A(x_j^1, \ldots, x_j^n))_{j=1}^\infty \in Y(F)$ whenever $(x_j^k)_{j= 1}^\infty \in X_k(E)$, $k = 1, \ldots, n$. The space of all such operators is denoted by $\Pi_{X_1, \ldots, X_n;Y}(E_1, \ldots, E_n;F)$. Under a certain technical condition (cf. Section 2), $\Pi_{X_1, \ldots, X_n;Y}$ is an ideal of $n$-linear operators (precise definitions will be given in Section 2).

In \cite[Proposition 2.4]{botelho+wood} it is proved that, for all sequence classes $X,Y$ and every $n \geq 2$, the ideal $\Pi_{X, \stackrel{(n)}{\ldots}, X;Y}=: \Pi_{^nX ;Y}$ is symmetric.
%By ${\cal L}^n$ we denote the class of all continuous $n$-linear operators between Banach spaces.
And it is easy to see that, regardless of the sequence classes $X_1, \ldots, X_n$, the ideal $\Pi_{X_1, \ldots, X_n;\ell_\infty(\cdot)}$ coincides with the class of all continuous $n$-linear operators, hence it is trivially symmetric. %In the linear case we just write $\cal L$.
So, we are chasing symmetric ideals of the type $\Pi_{X_1, \ldots, X_n;Y}$ such that $X_i \neq X_j$ for some $i \neq j$ and $Y \neq \ell_\infty(\cdot)$. We shall say that such an ideal is a non trivial symmetric sequential ideal. %Rephrasing, our purpose in this paper is to show how to create non-trivially symmetric ideals and to give concrete examples.?????}

The contribution of this paper is divided into three sections. Philosophically, non symmetric ideals should be the rule, whereas symmetric ideals should be the exceptions. On the one hand, in Section 3 we show that, indeed, there are many examples of non symmetric sequential ideals. On the other hand, the reader will have the opportunity to check that, somehow surprisingly, such ideals are not so easy to be found. Sections 4 and 5 have the same structure: first we prove a result that gives conditions for $\Pi_{X_1, \ldots, X_n;Y}$ to be symmetric, then we construct two procedures to create new sequence classes from a given sequence class, namely $X \mapsto X^u$ in Section 4 and $X \mapsto X^{\rm fd}$ in Section 5, which will enable us to give concrete examples of non trivial symmetric sequential ideals. In Section 6 we give additional applications of the procedures $X \mapsto X^u$ and $X \mapsto X^{\rm fd}$, one application in the theory of operator ideals and another one in the theory of sequence classes.

\section{Preliminaries}

For Banach spaces $E, E_1, \ldots, E_n,F$ over $\mathbb{K} = \mathbb{R}$ or $\mathbb{C}$, $E^*$ denotes the topological dual of $E$, $B_E$ denotes the closed unit ball of $E$ and ${\cal L}(E_1, \ldots, E_n;F)$ denotes the space of continuous $n$-linear operators from $E_1 \times \cdots \times E_n$ to $F$. If $E = E_1 = \cdots = E_n$, we write ${\cal L}(^n E;F)$. For the theory of continuous multilinear operators we refer to \cite{sean, jorge}. In the linear case $n = 1$ we simply write ${\cal L}(E;F)$. A Banach ideal of $n$-linear operators, see \cite{fg, pietsch}, is a subclass ${\cal M}_n$ of the class of continuous $n$-linear operators between Banach spaces such that, for all $E_1, \ldots, E_n,F$,
$${\cal M}_n(E_1, \ldots, E_n;F) := {\cal M}_n \cap  {\cal L}(E_1, \ldots, E_n;F)$$
is a linear subspace of ${\cal L}(E_1, \ldots, E_n;F)$ endowed with a complete norm $\|\cdot\|_{{\cal M}_n}$ such that:\\
$\bullet$ For all $x_1^* \in E_1^*, \ldots, x_n^* \in E_n^*$ and $y \in F$, ${\cal M}_n$ contains the operator
$$(x_1, \ldots, x_n)\in E_1 \times \cdots \times E_n \mapsto x_1^*(x_1) \cdots x_n^*(x_n)y \in F. $$
$\bullet$ $\| (\lambda_1, \ldots, \lambda_n) \in \mathbb{K}^n \mapsto \lambda_1 \cdots \lambda_n \in \mathbb{K}\|_{{\cal M}_n} = 1$.\\
$\bullet$ If $A \in {\cal M}_n(E_1, \ldots, E_n;F)$, $u_j \in {\cal L}(G_j;E_j)$, $j = 1, \ldots, n$ and $t \in {\cal L}(F;H)$, then $t \circ A \circ (u_1, \ldots, u_n) \in {\cal M}_n(G_1, \ldots, G_n;H)$ and
$$\|t \circ A \circ (u_1, \ldots, u_n)\|_{{\cal M}_n} \leq \|t\|\cdot \|A\|_{{\cal M}_n} \cdot \|u_1\| \cdots \|u_n\|. $$

The symmetrization of the $n$-linear operator $A \in {\cal L}(^n E;F)$ is the $n$-linear operator $A_s \in {\cal L}(^n E;F)$ given by
$$A_s(x_1, \ldots, x_m) = \frac{1}{n!}\cdot \sum_{\sigma \in S_n} A(x_{\sigma(1)}, \ldots, x_{\sigma(n)}), $$
where $S_n$ is the set of permutations of $\{1, \ldots, n\}$. An ideal ${\cal M}_n$ of $n$-linear operators is said to be {\it symmetric} \cite{fg} if $A_s$ belongs to ${\cal M}_n$ for every $A$ belonging to ${\cal M}_n$.

Since the beginning of the theory, that is, since \cite{pietsch}, ideals of generalized summing operators, have been playing a central role. Now we describe the approach from \cite{botelho+campos} to these ideals. The symbol $E \stackrel{1}{\hookrightarrow} F$ means that $E$ is a linear subspace of $F$ and $\|\cdot\|_F \leq \|\cdot\|_E$ on $E$. We write $E \stackrel{1}{=}$ if $E = F$ isometrically. By $c_{00}(E)$ and $\ell_\infty(E)$ we denote the spaces of eventually null and bounded $E$-valued sequences. For $j \in \mathbb{N}$, set $e_j := (0, \ldots, 0,1, 0,0, \ldots)$, where $1$ appears at the $j$-th coordinate.

A {\it sequence class} is a rule that assigns, to each Banach space $E$, a Banach space $X(E)$ of $E$-valued sequences such that $c_{00}(E) \subseteq X(E) \stackrel{1}{\hookrightarrow} \ell_\infty(E)$ and $\|e_j\|_{X(\mathbb{K})} = 1$ for every $j \in \mathbb{N}$.

\begin{example}\label{ey8n}\rm Let $1 \leq p < \infty$. The following correspondences are sequence classes: $E \mapsto \ell_\infty(E)$ = bounded $E$-valued sequences, $E \mapsto c_0(E)$ = norm null $E$-valued sequences, $E \mapsto c(E)$ = convergent $E$-valued sequences, $E \mapsto c_0^w(E)$ = weakly null $E$-valued sequences, $E \mapsto \ell_p(E)$ = absolutely $p$-summable $E$-valued sequences, $E \mapsto \ell_p^w(E)$ = weakly $p$-summable $E$-valued sequences, $E \mapsto \ell_p^u(E)$ = unconditionally $p$-summable $E$-valued sequences (see \cite[8.2]{df}), $E \mapsto {\rm Rad}(E)$ = almost unconditionally summable $E$-valued sequences (see \cite[Chapter 12]{diestel+jarchow+tonge}), $E \mapsto {\rm RAD}(E)$ = almost unconditionally bounded $E$-valued sequences (see \cite{vakhania}), $E \mapsto \ell_p\langle E \rangle$ = Cohen strongly $p$-summable $E$-valued sequences (see \cite{cohen}), $E \mapsto \ell_p^{\rm mid}(E)$ = mid $p$-summable $E$-valued sequences (see \cite{botelho+campos+santos}). All sequences spaces aforementioned are endowed with their natural norms. For details, see, e.g., \cite{botelho+campos+2}.
\end{example}

When referring to a sequence class $X$, we shall adopt the following {\it modus vivendi}: at any risk of ambiguity we shall write $X(\cdot)$, otherwise we simply write $X$. For example, we write $\ell_p^w$ instead of $\ell_p^w(\cdot)$, but we write $\ell_p(\cdot)$ and not $\ell_p$.

Let $X_1, \ldots, X_n,Y$ be sequence classes. An $n$-linear operator $A \in {\cal L}(E_1, \ldots, E_n;F)$ is said to be $(X_1, \ldots, X_n;Y)$-summing, in symbols $A \in \Pi_{X_1, \ldots, X_n;Y}(E_1, \ldots, E_n;F)$, if $(A(x_j^1, \ldots, x_j^n))_{j=1}^\infty \in Y(F)$ whenever $(x_j^k)_{j= 1}^\infty \in X_k(E_k), k = 1, \ldots, n$. In this case, the induce $n$-linear operator
$$\widehat{A} \colon X(E_1) \times \cdots \times X(E_n) \longrightarrow Y(F)~,~\widehat{A}(((x_j^1)_{j= 1}^\infty, \ldots, (x_j^n)_{j= 1}^\infty) =  (A(x_j^1, \ldots, x_j^n))_{j=1}^\infty,$$
is continuous and the expression $\|A\|_{\Pi_{X_1, \ldots, X_n;Y}} := \|\widehat{A}\|$ makes $\Pi_{X_1, \ldots, X_n;Y}(E_1, \ldots, E_n;F)$ a Banach space. A sequence class $X$ is {\it linearly stable} if, regardless of the $n \in \mathbb{N}$ and the Banach space $E,F$, $\Pi_{X;X} \stackrel{1}{=} {\cal L}(E;F)$. All sequence classes listed in Example \ref{ey8n} are linearly stable.

Under the two conditions below, $\Pi_{X_1, \ldots, X_n;Y}$ is a Banach ideal of $n$-linear operators (see \cite[Theorem 3.6]{botelho+campos}:\\
$\bullet$ $X_1, \ldots, X_n$ and $Y$ are linearly stable.\\
$\bullet$ $(\lambda_j^1\cdots \lambda_j^n)_{j=1}^\infty\in Y(\mathbb{K})$ and $\|(\lambda_j^1\cdots \lambda_j^n)_{j=1}^\infty\|_{Y(\mathbb{K})} \leq \prod\limits_{k=1}^n \|(x_j^k)_{k=1}^\infty\|_{X_k(\mathbb{K})}$ for all $(\lambda_j^k)_{j= 1}^\infty \in X_k(\mathbb{K}), k = 1, \ldots, n$.

Of course we are interested in the case that $\Pi_{X_1, \ldots, X_n;Y}$ is a Banach ideal of $n$-linear operators, so we will always suppose that the two conditions above hold. In particular, henceforth all sequence classes are supposed to be linearly stable.

\section{Non symmetric ideals}
In this section we show that, although they are not so easy to be constructed, there are plenty of non symmetric ideals of the type $\Pi_{X_1, \ldots, X_n;Y}$. It is not difficult to see that non symmetric ideals of bilinear operators can be extended to the general multilinear case, so we restrict ourselves to the bilinear case. By $A^t$ we denote the transpose of a bilinear operator $A$, that is, $A^t(x,y) = A(y,x)$. So, an ideal $\cal M$ of bilinear operators is symmetric if and only if $A^t \in \cal M$ whenever $A \in {\cal M}$.

A sequence class $X$ is said to be:\\
$\bullet$ {\it Subsequence invariant} if $(x_j)_{j=1}^\infty \in X(E)$ implies that every subsequence $(x_{j_k})_{k=1}^\infty \in X(E)$ and $\|(x_{j_k})_{k=1}^\infty\|_{X(E)}\leq \|(x_{j})_{j=1}^\infty\|_{X(E)}.$

%All sequence classes listed in Example \ref{ey8n} are subsequence invariant.

\noindent$\bullet$ {\it $S$-complete}, where $S \subseteq \mathbb{K}^\mathbb{N}$,  if $(\alpha_j x_j)_{j=1}^\infty \in X(E)$ whenever $( x_j)_{j=1}^\infty \in X(E)$ and $(\alpha_j)_{j=1}^\infty \in S$.

The following sequence classes are subsequence invariant: $c_0(\cdot), c(\cdot),c_0^w$ and, for $1 \leq p \leq \infty$, $\ell_p \langle \cdot \rangle, \ell_p(\cdot), \ell_p^{\rm mid}, \ell_p^u, \ell_p^w$.
All sequence classes listed in Example \ref{ey8n}, but RAD, Rad and $c(\cdot)$, are $\ell_\infty$-complete. Due to Kahane's contraction principle, RAD and Rad are $\ell_\infty$-complete in the real case $\mathbb{K} = \mathbb{R}$. It is clear that $c(\cdot)$ is $c_0$-complete, hence $\ell_p$-complete for any $1 \leq p < \infty$.

\begin{proposition}\label{cap3P06} Let $X, Y, W$ be sequence classes such that $X$ is subsequence invariant and $W(\mathbb{K})$-complete, $W$ is $\ell_\infty$-complete and the ideal of bilinear operators $\Pi_{W,X;Y}$ is symmetric. Then $\Pi_{X;Y}(E;F) \subseteq \Pi_{W;Y}(E,F)$ for every Banach space $F$ and every Banach space $E$ such that $X(E) \not\subseteq c_0^w(E)$.
\end{proposition}

\begin{proof} Let $E,F$ be Banach spaces with $X(E) \not\subseteq c_0^w(E)$. We can choose a sequence $(y_j)_{j=1}^\infty \in X(E)$ and a functional $x^* \in E^*$ such that $x^*(y_j) \not\longrightarrow 0$. Since $X$ is subsequence invariant, we can suppose that $|x^*(y_j)| \geq \varepsilon$ for every $j \in \mathbb{N}$ and some $\varepsilon > 0$. Given $u \in \Pi_{X;Y}(E;F)$, let us check that the continuous bilinear operator
$$A \colon E \times E \longrightarrow F~,~A(x,y) = x^*(x)u(y), $$
is $(W,X;Y)$-summing. For $(x_j)_{j=1}^\infty \in X(E)$ and $(w_j)_{j=1}^\infty \in W(E)$, $(x^*(w_j))_{j=1}^\infty \in W(\mathbb{K})$ by the linear stability of $W$, hence $(x^*(w_j)x_j)_{j=1}^\infty \in X(E)$ by the $W(\mathbb{K})$-completeness of $X$. The $(X;Y)$-summability of $u$ gives
$$(A(w_j,x_j))_{j=1}^\infty = (x^*(w_j)u(x_j))_{j=1}^\infty = (u(x^*(w_j)x_j))_{j=1}^\infty\in  Y(F),$$ proving that $A \in \Pi_{W,X;Y}(E,E;F)$. Therefore $A^t \in \Pi_{W,X;Y}(E,E;F)$ by the symmetry of $\Pi_{W,X;Y}$. To prove the $(W;Y)$-summability of $u$, let  $(w_j)_{j=1}^\infty \in W(E)$ be given. Since $ \left( \frac{1}{x^*(y_{j})} \right)_{j=1}^\infty\in \ell_\infty$ (recall that $\left|\frac{1}{x^*(y_{j})}\right|\leq \frac{1}{\varepsilon} $ for every $ j $), calling on the $\ell_\infty$-completeness of $W$ once again we get that $\left( \frac{1}{x^*(y_{j})} w_j\right)_{j=1}^\infty \in W(E)$. Using that $(y_j)_{j=1}^\infty \in X(E)$, we finish the proof:
\begin{align*} (u(w_j))_{j=1}^\infty&=\left(\frac{1}{x^*(y_{j})}x^*(y_{j}) u(w_j)\right)_{j=1}^\infty=\left(\frac{1}{x^*(y_{j})} A(y_{j}, w_j)\right)_{j=1}^\infty\\
	&=\left(\frac{1}{x^*(y_{j})} A^t(w_j, y_{j})\right)_{j=1}^\infty = \left(A^t\left(\frac{1}{x^*(y_{j})} w_j, y_{j}\right)\right)_{j=1}^\infty \in Y(F).
\end{align*}
\end{proof}

\begin{example}\label{Ex1} \rm Let $1  \leq p < \infty$. Taking $W = \ell_p^w$ and $X = Y = c(\cdot)$, we have that $c(\cdot)$ is subsequence invariant and $\ell_p$-complete, $\ell_p^w$ is $\ell_\infty$-complete and $c(E) \not\subseteq c_0^w(E)$ for every $E$. Using that $\ell_p^w \subseteq c_0^w$, it is easy to check that $\Pi_{\ell_p^w;c(\cdot)}=\Pi_{\ell_p^w;c_0(\cdot)}$ (in the literature, the operators belonging to this ideal are called $p$-convergent, see \cite{ardakani, castillo}). It is clear that this operator ideal is non-trivial, for instance, the identity operator on $c_0$ does not belong to it. We have
$$\Pi_{c(\cdot);c(\cdot)} = {\cal L} \not\subseteq \Pi_{\ell_p^w;c_0(\cdot)} =  \Pi_{\ell_p^w;c(\cdot)}, $$
where the first equality is the linear stability of $c(\cdot)$. Proposition \ref{cap3P06} gives that the ideal $\Pi_{\ell_p^w,c(\cdot);c(\cdot)}$  of $(\ell_p^w,c(\cdot);c(\cdot))$-summing bilinear operators is not symmetric.
\end{example}

\begin{proposition} Let $X$ be a subsequence invariant sequence class not contained in $c_0^w$. Then, for any $1 \leq p < \infty$, the ideal $\Pi_{\ell_p^w, X; \ell_p(\cdot)}$ is not symmetric.
\end{proposition}

\begin{proof} By assumption there are a Banach space $E$ and a non weakly null sequence $(z_j)_{j=1}^\infty \in X(E)$, say $z^*(z_j) \not\longrightarrow 0$, $z^* \in E^*$. Let us see that we can suppose that $E$ is infinite dimensional. If not, define $w_j = (z_j,0,0,0,\ldots) \in \ell_1(E)$ for every $j$. Considering the inclusion $x \in E \mapsto (x,0,0, \ldots) \in \ell_1(E)$, the linear stability of $X$ gives that $(w_j)_{j=1}^\infty \in X(\ell_1(E))$. By \cite[II.B.21]{woy}, $w^* := (z^*,0,0, \ldots) \in \ell_\infty(E^*) = \left[\ell_1(E)\right]^*$ and $w^*(w_j) = z^*(z_j) \not\longrightarrow 0$. So, $(w_j)_{j=1}^\infty $ is a non weakly null sequence in the infinite dimensional space $\ell_1(E)$ belonging to $X$. Therefore we can assume that $(z_j)_{j=1^\infty} \in X(E)$ is a non weakly null sequence in an infinite dimensional Banach space $E$.

Since $X$ is subsequence invariant, up to a subsequence we can take $\varepsilon > 0$ and $x^* \in E^*$ such that $|x^*(z_j)| \geq \varepsilon$ for every $j$. Using that $X(E)\stackrel{1}{\hookrightarrow}\ell_\infty(E)$, it follows easily that the bilinear operator
$$A \colon E \times E \longrightarrow E~,~A(x,y) = x^*(x)y, $$
is $(\ell_p^w, X; \ell_p(\cdot))$-summing, that is, $A\in \Pi_{\ell_p^w,X;\ell_p(\cdot)}(E,E;E)$. Since $E$ is infinite dimensional, from \cite[Theorem 2.18]{diestel+jarchow+tonge} there is $(x_j)_{j=1}^{\infty}\in \ell_p^w(E)\setminus \ell_p(E)$. So,
	\begin{align*}
	\sum_{j=1}^{\infty} \| A^t(x_j,z_j)\|^p&=\sum_{j=1}^{\infty} \| A(z_j,x_j)\|^p=\sum_{j=1}^{\infty} |x^*(z_j)|^p \cdot\|x_j\|^p\geq \varepsilon^p \cdot \sum_{j=1}^{\infty} \|x_j\|^p=+\infty,
	\end{align*}
proving that $A^t\notin \Pi_{\ell_p^w,X;\ell_p(\cdot)}(E,E;E)$.
\end{proof}

\begin{example}\rm For $1 \leq p < \infty$, the ideals $\Pi_{\ell_p^w, c(\cdot); \ell_p(\cdot)}$ and $\Pi_{\ell_p^w, \ell_\infty(\cdot); \ell_p(\cdot)}$ are not symmetric.
\end{example}

We finish this section with a concrete example which is not covered by the results above.

\begin{example}\rm Let $1 \leq p < q < \infty$. As we saw in Example \ref{Ex1}, we can take an infinite dimensional Banach space $E$ and a non-norm null weakly $p$-summable sequence $(x_j)_{j=1}^\infty \subseteq E$, that is, $(x_j)_{j=1}^\infty \in \ell_p^w(E) \setminus c_0(E)$. Up to a subsequence, we can suppose that $\|x_j\| \geq \varepsilon$ for every $j \in \mathbb{N}$ and some $\varepsilon > 0$. Pick a nonzero functional $x^* \in E^* $ and consider %  and pick $z \in E, \|z\| = 1$. Taking a functional $x^* \in E^*$ such that $x^*(z) = 1$, consider
%Using the Hahn-Banach theorem it is easy to see that $\ell_p^w(E) \subsetneqq \ell_q^w(E)$, so we can take $(z_j)_{j=1}^\infty \in \ell_q^w(E)\setminus \ell_p^w(E)$. Letting $x^* \in E^*$ be such that $(x^*(z_j))_{j=1}^\infty \in \ell_q\setminus \ell_p$, consider
the bilinear operator
	$$ A\colon E\times E \longrightarrow E ~,~A(x,y)=x^*(x)y.$$
On the one hand, it is easy to see that $A$ is $(\ell_p^w, \ell_q^w;\ell_p(\cdot))$-summing. On the other hand, let $(\alpha_j)_{j=1}^\infty$ be a $q$-summable non-$p$-summable scalar sequence, that is $(\alpha_j)_{j=1}^\infty \in \ell_q \setminus \ell_p$, and choose $z \in E$ such that $x^*(z) \neq 0$. Defining $z_j = \alpha_j z$ for every $j$, it is clear that $(z_j)_{j=1}^\infty \in \ell_q^w(E)$ and $(x^*(z_j))_{j=1}^\infty \notin\ell_p$. In this fashion, $(x_j)_{j=1}^\infty \in \ell_p^w(E), (z_j)_{j=1}^\infty \in \ell_q^w(E)$ and
\begin{align*}
	\sum_{j=1}^{\infty} \| A^t(x_{j},z_j)\|^p&=\sum_{j=1}^{\infty} \| A(z_j,x_{j})\|^p =\sum_{j=1}^{\infty} |x^*(z_j)|^p \cdot \|x_{j}\|^p \geq \varepsilon^p \cdot \sum_{j=1}^{\infty} |x^*(z_j)|^p=+\infty,
	\end{align*}
which proves that $A^t$ fails to be $(\ell_p^w, \ell_q^w;\ell_p(\cdot))$-summing. The conclusion is that the ideal $\Pi_{\ell_p^w,\ell_q^w;\ell_p(\cdot)}$ fails to be symmetric.
\end{example}

\section{The procedure $X \mapsto X^u$}
The purpose of this section and of the next one is not only to provide examples of non trivial symmetric ideals of the type $\Pi_{X_1, \ldots, X_n;Y}$, but also to show how to create such ideals.

In this section, first we give a general result, then concrete applications are given. Soon we shall see %In this section we prove two results showing how to create symmetric ideals of the type . In the next two sections we provide plenty of concrete applications of these results. In Section 5 we shall give
examples of sequence classes satisfying the following conditions:

\begin{definition}\rm (a) For the sequence classes $X$ and $Y$, we write $X\stackrel{\rm fin}{\leq} Y$ if, for every Banach space $E$, $\|\cdot\|_{X(E)} \leq \|\cdot\|_{Y(E)}$ on $c_{00}(E)$. If the norms coincide, we write  $X\stackrel{\rm fin}{=} Y$.
\\
(b) The sequence classes $X_1,\ldots,X_n$ are said to be {\it jointly dominated} if there exists a finitely determined sequence class $X$ such that $X_i\stackrel{1}{\hookrightarrow} X$ and $X_i\stackrel{\rm fin}{\leq} X$ for $i=1,\ldots,n$. %Neste caso, dizemos que as classes $X_1,\ldots,X_n$ são conjuntamente dominadas por $X$.
\end{definition}

\begin{theorem}\label{cap3T1.1}
	If the sequence classes $X_1,\ldots,X_n$ are jointly dominated, then the ideal $\Pi_{X_1,\ldots,X_n;Y}$ is symmetric for every finitely determined sequence class $Y$.
\end{theorem}
\begin{proof} Let $Y$ be a finitely determined sequence class and
	let $X$ be a finitely determined sequence class such that  $X_i\stackrel{1}{\hookrightarrow} X$ and $X_i\stackrel{\rm fin}{\leq} X$ for $i=1,\ldots,n$. By \cite[Proposition 2.4]{botelho+wood} it is enough to show that  $\Pi_{X_1,\ldots,X_n;Y}\stackrel{1}{=}\Pi_{^nX;Y}$. On the one hand, the condition $X_i\stackrel{1}{\hookrightarrow} X, i = 1, \ldots, n$, gives immediately that   $\mathcal{L}_{^nX;Y} \stackrel{1}{\hookrightarrow} \mathcal{L}_{X_1,\ldots,X_n;Y}$. %Dados  $A\in\mathcal{L}_{^nX;Y}(E_1,\ldots , E_n;F) $ e $(x_j^i)_{j=1}^\infty \in X_i(E_i)$, $ i=1,\ldots,n$,
	%\begin{align*}
%	\|(A(x_j^1,\ldots ,x_j^n))_{j=1}^{\infty}\|_{Y(F)}&\leq \|A\|_{^nX;Y}\prod_{i=1}^n \|(x_j^i)_{j=1}^\infty\|_{X(E_i)}\\
%	&\leq \|A\|_{^nX;Y}\prod_{i=1}^n \|(x_j^i)_{j=1}^\infty\|_{Xi(E_i)}.\\
%	\end{align*}
	On the other hand, given $A\in\mathcal{L}_{X_1,\ldots,X_n;Y}(E_1,\ldots , E_n;F) $,  $k\in\mathbb{N}$ and $x_j^i\in E_i$, $j=1,\ldots, k$, $i=1,\ldots, n$, the condition $X_i\stackrel{\rm fin}{\leq} X$ for $i=1,\ldots,n$ gives
	\begin{align*}
	\|(A(x_j^1,\ldots ,x_j^n))_{j=1}^{k}\|_{Y(F)}&\leq \|A\|_{X_1,\ldots,X_n;Y}\!\cdot\! \prod_{i=1}^n \|(x_j^i)_{j=1}^k\|_{X_i(E_i)} \leq \|A\|_{X_1,\ldots,X_n;Y}\!\cdot\! \prod_{i=1}^n \|(x_j^i)_{j=1}^k\|_{X(E_i)}.
	\end{align*}
Since both $X$ and $Y$ are finitely determined, taking the suprema over $k$ it follows that $A\in \mathcal{L}_{^nX;Y}(E_1,\ldots , E_n;F) $ and $\|A\|_{^nX;Y}\leq \|A\|_{X_1,\ldots,X_n;Y}$. Therefore, $\mathcal{L}_{X_1,\ldots,X_n;Y} \stackrel{1}{\hookrightarrow} \mathcal{L}_{^nX;Y}$.
\end{proof}

Now we proceed to give applications of the theorem above.

\begin{definition}\rm Let $X$ be a sequence class. \\
(a) For a Banach space $E$, we define $\overline{c_{00}}^X(E)$ as the closure of ${c_{00}(E)}$ in $X(E)$ endowed with the norm $\|\cdot\|_{X(E)}$.\\
(b) $X$ is said to be {\it finitely shrinking} if, regardless of the Banach space $E$, the sequence $(x_j)_{j=1}^\infty\in X(E)$ and $k\in \mathbb{N}$, it holds		$$(x_{j})_{j\neq k}: = (x_1, \ldots, x_{k-1}, x_{k+1}, \ldots )\in X(E) ~~\text{and}~~ \|(x_{j})_{j\neq k}\|_{X(E)}\leq \|(x_{j})_{j=1}^\infty\|_{X(E)}.$$
(c) Suppose that $X$ is finitely shrinking. If $(x_j)_{j=1}^\infty\in X(E)$, then $(x_n, x_{n+1}, \ldots) \in X(E)$ for every $n \in \mathbb{N}$. For a Banach space $E$, we define %the unconditional kernel of $X$ by
$$ X^u(E):=\left\{(x_j)_{j=1}^\infty\in X(E): \lim_n\|(x_n ,x_{n+1},\ldots)\|_{X(E)} = 0 \right\},$$
endowed with the norm $\|\cdot\|_{X(E)}$.
\end{definition}

\begin{example}\rm It is not difficult to check that $\overline{c_{00}}^{\ell_p^w}=(\ell_p^w)^u=\ell_p^u$ and $\overline{c_{00}}^{\ell_p(\cdot)}=(\ell_p(\cdot))^u=\ell_p(\cdot)$ for every $1 \leq p < \infty$, and that $\overline{c_{00}}^{\ell_\infty(\cdot)}=(\ell_\infty(\cdot))^u=c_0(\cdot)$. Using \cite[Theorem 3.10]{fourie+rontgen} we get
	$\overline{c_{00}}^{\ell_p\langle \cdot \rangle}=  (\ell_p\langle\cdot\rangle)^u=\ell_p\langle\cdot\rangle$. Additional examples will be given soon.
\end{example}

It is easy to check that $\overline{c_{00}}^X$ is a (linearly stable) sequence class. It is clear that subsequence invariant classes are finitely shrinking, so we already have plenty of examples of finitely shrinking classes. Soon we will show that RAD is finitely shrinking.

The procedure $X \mapsto X^u$ was first studied in the thesis \cite{teseLucas}, yet to be published. There it is proved, among other things and under different assumptions on  $X$, that $X^u$ is a (linearly stable) sequence class. Before proceeding, we have to go into details of the class $X^u$.

A sequence class $Y$ is a {\it closed subclass} of the sequence class $X$ if, for every Banach space $E$, $Y(E)$ is a closed subspace of $X(E)$ and $\|\cdot\|_{Y(E)} = \|\cdot\|_{X(E)}$ on $Y(E)$. Of course, $\overline{c_{00}}^X$ is a closed subclass of $X$.

The reader may be tempted to think that $\overline{c_{00}}^X = X^u$ for every $X$. But there is a subtle detail. In the next proposition we shall see that the equality holds if, in addition to be finitely shrinking, $X$ satisfies the following property: for all $E$, $(x_j)_{j=1}^\infty \subseteq E$ and $k \in \mathbb{N}$ for which $x_k = 0$, it holds
$$(x_j)_{j=1}^\infty \in X(E) \Leftrightarrow  (x_j)_{j\neq k} \in X(E) \mbox{~and~}\|(x_j)_{j=1}^\infty\|_{X(E)}=\|(x_j)_{j\neq k}\|_{X(E)}.$$
In this case  $X$ is said to be {\it finitely zero invariant}.
All finitely shrinking classes we have worked with are finitely zero invariant, but we do not know if this holds in general.

Next we summarize the properties of $X^u$ we shall use later. We skip the proofs.

\begin{proposition}\label{propxu} Let $X$ and $Y$ be a finitely shrinking sequence classes.\\
{\rm (a)} $X^u$ is a (linearly stable) finitely shrinking sequence class, $(X^u)^u = X^u$ and $X^u \stackrel{1}{\hookrightarrow} c_0(\cdot)$. \\
{\rm (b)} If $X$ is finitely zero invariant, then $X^u = \overline{c_{00}}^X$. \\
{\rm (c)} $X^u$ and $\overline{c_{00}}^X$ are closed subclasses of $X$. In particular, $X^u \stackrel{1}{\hookrightarrow} X$ and $\overline{c_{00}}^X \stackrel{1}{\hookrightarrow} X$.\\
{\rm (d)} If $X\stackrel{1}{\hookrightarrow} Y$ then $X^u \stackrel{1}{\hookrightarrow} Y^u$.
\end{proposition}

%The next result is a straightforward combination of the proposition above with Theorem

\begin{corollary}\label{pl9h} Let $X$ and $Y$ be finitely determined sequence classes with $X$ finitely shrinking. The ideal $\Pi_{X_1,\ldots,X_n;Y}$ is symmetric for every $n \geq 2$ and all $X_i \in \{\overline{c_{00}}^X ,X^u,X\}$, $i=1,\ldots,n$.
\end{corollary}

\begin{proof} By Proposition \ref{propxu}(a) and (c) we have that $X_1, \ldots, X_n$ are jointly dominated by $X$. The result follows from Theorem  \ref{cap3T1.1}.
\end{proof}

We are ready to give the first examples of non trivial sequential symmetric ideals.

\begin{example}\rm Let $1 \leq p < \infty$. Since $\ell_p^w$ is finitely determined and $(\ell_p^w)^u = \ell_p^u \neq \ell_p^w$, $\Pi_{\ell_p^{\theta_1},\ldots, \ell_p^{\theta_n};Y }$ is a non trivial sequential symmetric ideal for every $n \geq 2$, all $\theta_1, \ldots, \theta_n \in \{u,w\}$ with $\theta_i = u$ and $\theta_j = w$ for some $i$ and $j$, and every finitely determined sequence class $Y \neq \ell_\infty(\cdot)$. %As particular instances, $\Pi_{\ell_p^w, \ell_p^u; \ell_p^w}$ and $\Pi_{\ell_p^w, \ell_p^u; \ell_p(\cdot)}$ are non trivial sequential symmetric ideals. % to check that $\Pi_{\ell_p^w, \ell_p^u; \ell_p^w} \neq {\cal L}^2$. Note that in the proof of Theorem \ref{cap3T1.1} we proved that $\Pi_{\ell_p^w, \ell_p^u; \ell_p^w} = \Pi_{\ell_p^w, \ell_p^w; \ell_p^w}$. Using that $\ell_p^w$ is not multilinearly symmetric \cite[Theorem ???]{botelho+campos}, we have $\Pi_{\ell_p^w, \ell_p^u; \ell_p^w} \neq {\cal L}^2$, that is, this is a non trivial sequential symmetric ideal of bilinear operators. Since $\ell_p(\cdot)$ is finitely determined and  $\Pi_{\ell_p^w, \ell_p^u; \ell_p(\cdot)} \subseteq \Pi_{\ell_p^w, \ell_p^u; \ell_p^w} $, $\Pi_{\ell_p^w, \ell_p^u; \ell_p(\cdot)}$ is a non trivial sequential symmetric ideal of bilinear operators as well.
\end{example}

Some preparation is needed to give further concrete applications of Corollary \ref{pl9h}.

\begin{lemma}\label{lemarad} The sequence class {\rm RAD} is finitely shrinking and finitely zero invariant.
\end{lemma}

\begin{proof} Let $(x_j)_{j=1}^\infty \subseteq E$ and $k \in \mathbb{N}$ be so that $x_k = 0$. Given $n$, the symbol $\varepsilon_1,\stackrel{[k]}{\ldots},\varepsilon_{n}$ means $\varepsilon_1,\ldots, \varepsilon_{k-1}, \varepsilon_{k+1}, \ldots \varepsilon_{n}$. We have
\begin{align*}
	&\|( x_1,\ldots,x_{k-1},x_{k+1},\ldots, x_n,0,0,\ldots)\|^2_{{\rm Rad}(E)}\\
		&= \frac{1}{2^{n-1}}\cdot \sum_{\varepsilon_1,\ldots,\varepsilon_{n-1}=\pm 1} \|\varepsilon_1x_1+\cdots+\varepsilon_{k-1}x_{k-1} +\varepsilon_{k}x_{k+1}+\cdots+ \varepsilon_{n-1}x_n \|^2 \\
	&= \frac{1}{2^{n}} \cdot \sum_{\varepsilon_1,\stackrel{[k]}{\ldots},\varepsilon_{n}=\pm 1} 2\|\varepsilon_1x_1+\cdots+\varepsilon_{k-1}x_{k-1} +\varepsilon_{k+1}x_{k+1}+\cdots+ \varepsilon_{n}x_n \|^2\\
	%&= \left(\frac{1}{2^{n}} \sum_{\varepsilon_1,\stackrel{[k]}{\ldots},\varepsilon_{n}=\pm 1} 2\|\varepsilon_1x_1+\cdots+\varepsilon_{k-1}x_{k-1} +\varepsilon_{k+1}x_{k+1}+\cdots+ \varepsilon_{n}x_n \|^2 \right)^{1/2}\\
	&= \frac{1}{2^{n}} \cdot \sum_{\varepsilon_1,\stackrel{[k]}{\ldots},\varepsilon_{n}=\pm 1} \sum_{\varepsilon_k=\pm 1}\|\varepsilon_1x_1+\cdots+\varepsilon_{k-1}x_{k-1}+\varepsilon_kx_k +\varepsilon_{k+1}x_{k+1}+\cdots+ \varepsilon_{n}x_n \|^2\\
	&= \frac{1}{2^{n}} \cdot\sum_{\varepsilon_1,\ldots,\varepsilon_{n}=\pm 1} \|\varepsilon_1x_1+\cdots+ \varepsilon_{n}x_n \|^2%= \int_{0}^{1} \|  r_1(t)x_1+\cdots+ r_{n}(t)x_n \|^2 dt\right)^{1/2}\\
%	&= \left(\int_{0}^{1} \|  r_1(t)x_1+\cdots+ r_{n}(t)x_n \|^2 dt\right)^{1/2}\\
	= \|( x_1,\ldots, x_n,0,0,\ldots)\|^2_{{\rm Rad}(E)}.
	%&= \|( x_1,\ldots,x_{k-1},0,x_{k+1},\ldots,)\|_{\text{RAD}(E)},\\
	\end{align*}
Taking the supremum over $n$ we conclude that RAD is finitely zero invariant. Combining this information with Kahane's contraction principle \cite[12.2]{diestel+jarchow+tonge}, it follows easily that RAD is finitely shrinking.
\end{proof}

\begin{proposition} {\rm (a)} ${\rm RAD}^u = {\rm Rad} = {\rm Rad}^u$.\\
{\rm (b)} For $1 \leq p < \infty$, $(\ell_p^{\rm mid})^u\neq \ell_p^{\rm mid}$ and  $(\ell_p^{\rm mid})^u\neq \ell_p^u$.
\end{proposition}

\begin{proof} (a) We have
$$\text{RAD}^u=\overline{c_{00}}^{\text{RAD}} = {\rm Rad} = \text{RAD}^u = (\text{RAD}^u)^u = {\rm Rad}^u,$$
where the first equality follows from a combination of Lemma \ref{lemarad} with Proposition \ref{propxu}(b), the second from  %$\text{RAD}^u=\overline{c_{00}}^{\text{RAD}}(\cdot)$. According to
\cite[Proposition 5.1(b)]{vakhania}, the third has already been proved, the fourth follows from Proposition \ref{propxu}(a) and the last one from the already proved equality ${\rm RAD}^u = {\rm Rad}$.\\%, $\overline{c_{00}}^{\text{RAD}(\cdot)}(\cdot)= \text{Rad}(\cdot) $. Portanto $\text{RAD}^u(\cdot)=\overline{c_{00}}^{\text{RAD}(\cdot)}(\cdot)= \text{Rad}(\cdot).$ \\
(b) Since $\ell_p^{\rm mid} \stackrel{1}{\hookrightarrow} \ell_p^w$, by Proposition \ref{propxu}(d) we have $(\ell_p^{\rm mid})^u \stackrel{1}{\hookrightarrow} (\ell_p^w)^u=\ell_p^u.$ Supposing that $(\ell_p^{\rm mid})^u= \ell_p^{\rm mid}$ or $(\ell_p^{\rm mid})^u= \ell_p^u$, we would have $\ell_p^{\rm mid}=(\ell_p^{\rm mid})^u \subseteq \ell_p^u$ or
	$\ell_p^u=(\ell_p^{\rm mid})^u \subseteq \ell_p^{\rm mid}$; but this cannot happen because the classes $\ell_p^{\rm mid}$ and $\ell_p^u$ are incomparable \cite[Example 1.7]{botelho+campos+santos}.
\end{proof}

The following examples follow from the proposition above and Corollary \ref{pl9h}.

\begin{example}\rm Since RAD is finitely determined, for every $n \geq 2$, every finitely determined sequence class $Y \neq \ell_\infty(\cdot)$ and all sequences classes $X_1, \ldots, X_n \in \{{\rm RAD}, {\rm Rad}\}$ with $X_i = {\rm RAD}$ and $X_j = {\rm Rad}$ for some $i$ and $j$, $\Pi_{X_1, \ldots, X_n;Y}$ is  a non trivial sequential symmetric ideal.
\end{example}

\begin{example}\rm Let $1 \leq p < \infty$. Since $\ell_p^{\rm mid}$ is finitely determined, $\Pi_{X_1, \ldots, X_n;Y }$ is a non trivial sequential symmetric ideal for every $n \geq 2$, all $X_1, \ldots, X_n \in \{\ell_p^{\rm mid}, (\ell_p^{\rm mid})^u\}$ with $X_i = \ell_p^{\rm mid}$ and $X_j = (\ell_p^{\rm mid})^u$ for some $i$ and $j$, and every finitely determined sequence class $Y \neq \ell_\infty(\cdot)$.
\end{example}

\section{The procedure $X \mapsto X^{\rm fd}$}

%This section has a twofold purpose. First
Note that the symmetric ideals we constructed in the previous section are of the type $\Pi_{X_1, \ldots, X_n;Y }$ where $Y$ is always finitely determined. The good news is that several usual sequence classes are finitely determined, such as $\ell_p(\cdot), \ell_p^w, \ell_\infty$, RAD, $\ell_p\langle\cdot\rangle$ and $\ell_p^{\rm mid}$; and the bad news is that some are not, such as $c_0(\cdot), c_0^w$, $c(\cdot), \ell_p^{u}$, Rad and $(\ell_p^{\rm mid})^u$. The purpose of this section is to construct non trivial sequential symmetric ideals for $Y$ not necessarily finitely determined. %Until now we have not shown the existence a non trivial symmetric ideal of the type $\Pi_{X_1, X_2, X_3;Y}$ for which the classes $X_1, X_2, X_3$ are pairwise distinct. The second purpose of this section is to show that such ideals do exist. %Of course, other conditions will be imposed on the underlying sequence classes, but examples not covered by the previous section will be obtained.

%\begin{definition}\rm closed subclass
%\end{definition}

\begin{proposition} \label{cap3T01} Let $X_1,\ldots,X_n,Y,Z,W$ be sequence classes such that $Z$ and $W$ are finitely determined and finitely shrinking. %class of $W$ and $X_i \stackrel{1}{\hookrightarrow} Z$, $i=1,\ldots,n$.  %são finitamente contráteis. Suponha que;
Suppose also that
%	\begin{enumerate}[label=$\bullet$]
		$W\stackrel{\rm fin}{\leq} Y$, $X_i \subseteq Z$ and $X_i\stackrel{\rm fin}{\leq}Z$ for every $i=1,\ldots ,n$,
		$X_k \subseteq Z^u$ for some $k\in\{1,\ldots,n\}$ and $W^u\subseteq Y.$ Then $\Pi_{X_1,\ldots,X_n;Y} =  \Pi_{^nZ;W}$. %\stackrel{}{=}\mathcal{L}_{Z_1,\ldots,Z_n;W}.$$
%	Se $Y$ for subclasse fechada de $W$ e $X_i \stackrel{1}{\hookrightarrow} Z_i, i=1,\ldots,n$ então vale a igualdade de normas.
%\end{teorema}
\end{proposition}

\begin{proof} Let $A\in\Pi_{X_1,\ldots,X_n;Y}(E_1,\ldots,E_n;F) $ be given. Using that $W\stackrel{\rm fin}{\leq} Y$ and $X_i\stackrel{\rm fin}{\leq}Z$, for $m\in\mathbb{N}$ and $x_j^i\in E_i$, $j=1,\ldots, m$, $i=1,\ldots, n$, we have
	\begin{align*}
	\|(A(x_j^1,\ldots ,x_j^n))_{j=1}^{m}\|_{W(F)}&\leq\|(A(x_j^1,\ldots ,x_j^n))_{j=1}^{m}\|_{Y(F)}\leq \|A\|_{X_1,\ldots,X_n;Y}\prod_{i=1}^n \|(x_j^i)_{j=1}^m\|_{X_i(E_i)}\\
	&\leq \|A\|_{X_1,\ldots,X_n;Y}\prod_{i=1}^n \|(x_j^i)_{j=1}^m\|_{Z(E_i)}.
	\end{align*}
%	para todo $m\in\mathbb{N}$ e para todos $x_j^i\in E_i$, $j=1,\ldots, m$, $i=1,\ldots, n$.
Since $Z$ and $W$ are finitely determined, it follows that $A\in \Pi_{^nZ;W}(E_1,\ldots,E_n;F) $ and $\|A\|_{^nZ;W}\leq \|A\|_{X_1,\ldots,X_n;Y}$.

Now suppose that $A\in \Pi_{^nZ;W}(E_1,\ldots,E_n;F) $. Let $ (x_j^i)_{j=1}^\infty \in X_i(E_i)\subseteq  Z(E_i), i=1\ldots n $. %Devemos mostrar que $ (A(x_{j}^1,\ldots ,x_{j}^n))_{j=1}^{\infty} \in Y(F)$.
	Using that $Z$ is finitely shrinking, we get $(x_j^i)_{j=m}^\infty \in Z(E_i)$ and  $\|(x_j^i)_{j=m}^\infty\|_{Z(E_i)}\leq \|(x_j^i)_{j=1}^\infty\|_{Z(E_i)}$ for all $m\in\mathbb{N}$ and $i=1\ldots n$. Thus, $(A(x_{j}^1,\ldots ,x_{j}^n))_{j=m}^{\infty} \in W(F)$ and, since $(x_j^k)_{j=m}^\infty \in Z^u(E_k)$,
	\begin{align*}
	\lim_{m\to\infty}\|(A(x_{j}^1,& \ldots , x_{j}^n))_{j=m}^{\infty}  \|_{W(F)} \leq \lim_{m\to\infty} \|A\|_{^nZ;W}\cdot \prod_{i=1}^n  \|(x_{j}^i)_{j=m}^{\infty}\|_{Z(E_i)}\\
	&\leq \|A\|_{^nZ;W}\cdot \lim_{m\to\infty} \left[\left( \prod_{i=1,i\neq k}^n  \|(x_{j}^i)_{j=1}^{\infty}\|_{Z(E_i)}\right) \cdot  \|(x_{j}^k)_{j=m}^{\infty}\|_{Z(E_k)}\right] = 0.
	\end{align*}
    %$$\lim_{m\to\infty} \|(A(x_{j}^1,\ldots ,x_{j}^n))_{j=m}^{\infty}  \|_{W(F)}=0$$
This proves that $ (A(x_{j}^1,\ldots ,x_{j}^n))_{j=1}^{\infty} \in W^u(F) \subseteq Y(F)$, so $A\in \Pi_{X_1,\ldots,X_n;Y}(E_1,\ldots,E_n;F) $. The proof is complete. %We have just established the equality , from which the result follows. %, pois $X_k\subseteq Z_k^u$  e $W^u\subseteq Y$.
%Suponha que $Y$ seja subclasse fechada de $W$ e que $X_i \stackrel{1}{\hookrightarrow} Z_i, i=1,\ldots,n$. Para todos $(x_j^1)_{j=1}^\infty \in X_1(E_1), \ldots , (x_j^n)_{j=1}^\infty \in X_n(E_n) $
%	\begin{align*}
%	\|(A(x_{j}^1,\ldots ,x_{j}^n))_{j=1}^{\infty}  \|_{Y(F)} &= \|(A(x_{j}^1,\ldots ,x_{j}^n))_{j=1}^{\infty}  \|_{W(F)}\\
%	&\leq \|A\|_{Z_1,\ldots,Z_n;W} \prod_{i=1}^n  \|(x_{j}^i)_{j=1}^{\infty}\|_{Z_i(E_i)}\\
%	&\leq \|A\|_{Z_1,\ldots,Z_n;W} \prod_{i=1}^n  \|(x_{j}^i)_{j=1}^{\infty}\|_{X_i(E_i)},\\
%	\end{align*}
%	logo $ \|A\|_{X_1,\ldots,X_n;Y}\leq \|A\|_{Z_1,\ldots,Z_n;W}$.
\end{proof}

Next we show how the proposition above can be used to give new concrete examples of non trivial sequential symmetric ideals.

\begin{definition}\rm For a sequence class $X$ we define
	$$ X^{\rm fd}(E):=\left\{ (x_j)_{j=1}^\infty \in E^{\mathbb{N}}:\|(x_j)_{j=1}^\infty \|_{X^{\rm fd}(E)}: =\sup_{k}\|(x_j)_{j=1}^k\|_{X(E)}<\infty \right\} $$
for every Banach space $E$.
\end{definition}

\begin{example}\rm (a) For $1 \leq p < \infty$, $(\ell_p^u)^{\rm fd} = \ell_p^w = (\ell_p^w)^u$, $\ell_p(\cdot)^{\rm fd} = \ell_p(\cdot)$.

\medskip

\noindent (b) $c_0(\cdot)^{\rm fd} = (c_0^w)^{\rm fd} =  c(\cdot)^{\rm fd} =  \ell_\infty(\cdot)$ and ${\rm Rad}^{\rm fd} = {\rm RAD}$.

\medskip
\noindent(c) There are sequence classes $X$ for which $X^u \varsubsetneq X \varsubsetneq X^{\rm fd}$. For instance, taking $X = c_0^w$, we have
$$ (c_0^w)^u=c_0(\cdot) \varsubsetneq c_0^w \varsubsetneq \ell_\infty(\cdot) = (c_0^w)^{\rm fd}.$$
\end{example}

The properties of $X^{\rm fd}$ we shall need are the following.

\begin{proposition}\label{propxfd} Let $X$ be a sequence class. Then:\\
{\rm (a)} $X^{\rm fd}$ is a (linearly stable) sequence class and $X^{\rm fd} \stackrel{1}{=} X$ if and only if $X$ is finitely determined.\\
{\rm (b)} If $X$ is finitely shrinking, then $X^{\rm fd}$ is finitely shrinking and finitely determined.\\
{\rm (c)} If $X$ is finitely zero invariant, then $X^{\rm fd}$ is finitely zero invariant as well.\\
{\rm (d)} If $X$ is finitely shrinking and finitely zero invariant, then $\overline{c_{00}}^X \stackrel{1}{=} X^u \stackrel{1}{=} (X^{\rm fd})^u  \stackrel{1}{=} \overline{c_{00}}^{X^{\rm fd} }$. In particular, $\overline{c_{00}}^X = X^u$ is a closed subclass of $X^{\rm fd}$.
%{\rm (d)} If $X$ is finitely zero invariant, then $X^u = \overline{c_{00}}^X$ is a closed subclass of $X^{\rm fd}$.
\end{proposition}

\begin{proof} We skip the proof of (a). The completeness of $X^{\rm fd}(E)$ is proved by a standard argument with the help of \cite[Lemma 3.1]{botelho+campos}.\\
(b) Given a Banach space $E$, $k\in\mathbb{N}$ and $x_1,x_2,\ldots, x_k \in E$,
	\begin{equation} \|(x_j)_{j=1}^k\|_{X(E)}=\text{máx}\{\|(x_j)_{j=1}^1\|_{X(E)},\|(x_j)_{j=1}^2\|_{X(E)}, \ldots, \|(x_j)_{j=1}^k\|_{X(E)} \}= \|(x_j)_{j=1}^k\|_{X^{\rm fd}(E)} \label{d4eb}
\end{equation}
because $X$ is finitely shrinking. It follows that
	$$ (x_j)_{j=1}^\infty\in X^{\rm fd}(E) \Leftrightarrow \sup_k \|(x_j)_{j=1}^k\|_{X(E)}<\infty \Leftrightarrow \sup_k \|(x_j)_{j=1}^k\|_{X^{\rm fd}(E)}<+\infty, $$
and, in this case,
$$\|(x_j)_{j=1}^\infty\|_{X^{\rm fd}(E)}=\sup_k \|(x_j)_{j=1}^k\|_{X(E)}=\sup_k \|(x_j)_{j=1}^k\|_{X^{\rm fd}(E)},$$
proving that $X^{\rm fd}$ is finitely determined.

For $(x_j)_{j=1}^\infty\in X^{\rm fd}(E)$ and  $k\in \mathbb{N}$, using again that $X$ is finitely shrinking we get % . Basta olhar para a seguinte desigualdade
	$$ \|(x_1,\ldots,x_{k-1},x_{k+1},\ldots,x_n,0,0,\ldots)\|_{X(E)}\leq \|(x_j)_{j=1}^n\|_{X(E)}$$
for every $n > k$. Combining this with (\ref{d4eb}) it follows immediately that $X^{\rm fd}$ is finitely shrinking.\\
(c) Let $(x_j)_{j=1}^\infty \in E^\mathbb{N}$ and $k \in \mathbb{N}$ be so that $x_k = 0$. Since $X$ is finitely zero invariant,
\begin{equation}\|(x_j)_{j=1}^n \|_{X(E)} = \|(x_1, \ldots, x_{k-1}, x_{k+1}, \ldots, x_n, 0,0,\ldots)\|_{X(E)} \mbox{ for every } n > k. \label{fs4m}\end{equation}
Define $(y_j)_{j=1}^\infty$ by $y_j = x_j$ if $j < k$ and $y_j = x_{j+1}$ if $j \geq k$, that is, $(y_j)_{j=1}^\infty = (x_j)_{j \neq k}$. Combining (\ref{fs4m}) with
$$(x_1, \ldots, x_{k-1},0,0, \ldots) = (x_1, \ldots, x_{k-1}, x_k ,0 ,0, \ldots), $$  it is not difficult to see that
$$\sup_n \|(y_j)_{j=1}^n\|_{X(E)} =  \sup_n \|(x_j)_{j=1}^n\|_{X(E)}.$$
We conclude that $(x_j)_{j \neq k} = (y_j)_{j=1}^\infty \in X^{\rm fd}(E)$ if and only if $(x_j)_{j=1}^\infty \in X^{\rm fd}(E)$, and, in this case, $\|(x_j)_{j \neq k}\|_{X^{\rm fd}(E)} = \|(y_j)_{j=1}^\infty\|_{X^{\rm fd}(E)} \in X^{\rm fd}(E)$.  \\
(d) Proposition \ref{propxu} gives that $\overline{c_{00}}^X = X^u$ is a closed subclass of $X$. Let us check that $X^u$ is a closed subclass of $X^{\rm fd}$. Given $(x_j)_{j=1}^\infty \in X^u(E)$ and $\varepsilon>0$, let $k_0\in\mathbb{N}$ be such that $\|(x_k,x_{k+1},\ldots)\|_{X(E)}<\varepsilon$ for every $k\geq k_0$. Using that $X$ finitely shrinking and finitely zero invariant, we have
\begin{align*}
	\sup_{k} \|(x_j)_{j=1}^k\|_{X(E)}&=\sup_{k_0<k} \|(x_j)_{j=1}^k\|_{X(E)}=\sup_{k_0<k} \|(x_j)_{j=1}^\infty -  (\underbrace{0,\ldots,0}_{k-vezes},x_{k+1},x_{k+2},\ldots)   \|_{X(E)}\\
	&\leq\sup_{k_0<k} \left(\|(x_j)_{j=1}^\infty \|_{X(E)}+ \|(\underbrace{0,\ldots,0}_{k-vezes},x_{k+1},x_{k+2},\ldots)   \|_{X(E)}\right)\\
	&= \|(x_j)_{j=1}^\infty \|_{X(E)}+ \sup_{k_0<k} \|(\underbrace{0,\ldots,0}_{k-vezes},x_{k+1},x_{k+2},\ldots)   \|_{X(E)}\\
	&= \|(x_j)_{j=1}^\infty \|_{X(E)}+ \sup_{k_0<k} \|(x_{k+1},x_{k+2},\ldots)   \|_{X(E)}\leq  \|(x_j)_{j=1}^\infty \|_{X(E)}+ \varepsilon,
\end{align*}
which gives $(x_j)_{j=1}^\infty \in X^{\rm fd}(E)$ and $\|(x_j)_{j=1}^\infty\|_{ X^{\rm fd}(E)} \leq \|(x_j)_{j=1}^\infty\|_{X(E)} = \|(x_j)_{j=1}^\infty\|_{X^u(E)}$. Moreover, for every $k \geq k_0$,
\begin{align*}
	\|(x_j)_{j=1}^\infty \|_{X(E)}&= \|(x_j)_{j=1}^k + (\underbrace{0,\ldots,0}_{k-vezes},x_{k+1},x_{k+2},\ldots)  \|_{X(E)}\\
	&\leq \|(x_j)_{j=1}^k\|_{X(E)} + \|(\underbrace{0,\ldots,0}_{k-vezes},x_{k+1},x_{k+2},\ldots)  \|_{X(E)}\\
	&= \|(x_j)_{j=1}^k\|_{X(E)} + \|(x_{k+1},x_{k+2},\ldots)  \|_{X(E)} < \|(x_j)_{j=1}^k\|_{X(E)} + \varepsilon,% \\
%	&\leq \sup_{k} \|(x_j)_{j=1}^k\|_{X(E)} + \varepsilon,
\end{align*}
%para todo $\varepsilon >0$. Portanto $(x_j)_{j=1}^\infty \in X^{fd}(E)$ e
from which it follows that
$$\|(x_j)_{j=1}^\infty\|_{X^u(E)} = \|(x_j)_{j=1}^\infty\|_{X(E)}\leq \sup_k \|(x_j)_{j=1}^k\|_{X(E)} = \|(x_j)_{j=1}^\infty\|_{X^{\rm fd}(E)}$$ and proves that $X^u$ is a closed subclass of $X^{\rm fd}$.

Since $X$ (assumption) and $X^{\rm fd}$ (items (b) and (c)) are finitely shrinking and finitely zero invariant, Proposition \ref{propxu}(b) gives $X^u \stackrel{1}{=} \overline{c_{00}}^X$ and $(X^{\rm fd})^u \stackrel{1}{=} \overline{c_{00}}^{X^{\rm fd}}$. We have just proved that $X^u \stackrel{1}{=} \overline{c_{00}}^X$ is a closed subclass of $X^{\rm fd}$, in particular $X^u \stackrel{1}{\hookrightarrow} X^{\rm fd}$ with equal norms on $X^u$. Calling on Proposition \ref{propxu}(a) and (d) we get
$$ \overline{c_{00}}^X\stackrel{1}{=}X^u\stackrel{1}{=}(X^u)^u \stackrel{1}{\hookrightarrow} (X^{\rm fd})^u\stackrel{1}{=}\overline{c_{00}}^{X^{\rm fd}}$$
with equal norms in the inclusion. It is enough to prove that $\overline{c_{00}}^{X^{\rm fd}} \subseteq \overline{c_{00}}^X$. Given $(x_j)_{j=1}^\infty\in \overline{c_{00}}^{X^{\rm fd}}(E)$, using that $X^{\rm fd}$ is finitely zero invariant and $(X^{\rm fd})^u(E)=\overline{c_{00}}^{X^{\rm fd}}(E)$,
	\begin{align*}
	\|(x_j)_{j=1}^\infty- (x_j)_{j=1}^n\|_{X^{\rm fd}(E)}%&=\|(0,\ldots,0,x_{n+1},x_{n+2},\ldots )\|_{X^{\rm fd}(E)}
\stackrel{}{=}\|(x_{n+1},x_{n+2},\ldots )\|_{X^{\rm fd}(E)} \stackrel{n \to \infty}{\longrightarrow} 0,
	\end{align*}
proving that $(x_j)_{j=1}^n \stackrel{n \to \infty}{\longrightarrow} (x_j)_{j=1}^\infty$ in $X^{\rm fd}(E)$. By (\ref{d4eb}) we know that the norms $\|\cdot\|_{X^{\rm fd}}$ and $\|\cdot\|_X$ coincide on $c_{00}(E)$, so $((x_j)_{j=1}^n )_{n=1}^\infty$ is a Cauchy sequence in $X(E)$, hence convergent, say $(x_j)_{j=1}^n \stackrel{n \to \infty}{\longrightarrow} (y_j)_{j=1}^\infty$ in $X(E)$. But $\overline{c_{00}}^X(E)$ is closed in $X(E)$, so % segue que
%	$$  (y_j)_{j=1}^\infty \in \overline{c_{00}}^X(E),$$
%	ou seja,
$(x_j)_{j=1}^n \stackrel{n \to \infty}{\longrightarrow} (y_j)_{j=1}^\infty$ in $\overline{c_{00}}^X(E)$. This convergence occurs in $X^{\rm fd}(E)$ because $\overline{c_{00}}^X$ is a closed subclass of $X^{\rm fd}$. The uniqueness of the limit gives $(x_j)_{j=1}^\infty = (y_j)_{j=1}^\infty \in \overline{c_{00}}^X(E). $
	%quando $n\to\infty$ pois $(X^{fd})^u(E)=\overline{c_{00}}^{X^{fd}}(E)$.
%Proposition \ref{propxu} gives that $\overline{c_{00}}^X = X^u$ is a closed subclass of $X$. By (b) and, again, Proposition \ref{propxu}, we get that $(X^{\rm fd})^u$ is a closed subclass of $X^{\rm fd}$. It is enough to show that \textcolor{red}{$X^u \stackrel{1}{=} (X^{\rm fd})^u$}.
%\textcolor{blue}{Ariel: aqui eu preciso da sua ajuda. Por favor, escreva aqui a demonstração desse item. No arquivo da tese, isso está na Proposição 3.29(ii) e (iii) e na Proposição 3.31(vi). Eu sei que vai ter que colar argumentos de várias demonstrações do arquivo da tese. Para mim está difícil montar esse quebra-cabeça, por favor faça isso. Veja que vc pode usar a Proposition \ref{propxu}(b) deste arquivo e o item (b) desta proposição, provado acima. Pode escrever em português mesmo, depois eu passo para o inglês. Veja que não estou usando o conceito de finitamente zero relativo. Se precisar disso, use a propriedade sem dar o nome. Vc pode dizer que determinadas conclusões são óbvias (ou fáceis), não precisa demonstrar cada passo. Pode dizer coisas do tipo: ... Combinando isso que acabamos de provar com o fato de $X$ ser finitamente contrátil e finitamente zero invariante, segue facilmente que ...}
\end{proof}

\begin{theorem}	Let $X$ and $Y$ be finitely shrinking  and finitely zero invariant sequence classes. For every $n \in \mathbb{N}$, the ideal % $X$ for finitamente zero relativo e $Y$ finitamente zero invariante então, para todo $n\in\mathbb{N}$,
$\Pi_{X_1,\ldots,X_n;Y}$ is symmetric whenever $X_i\in \{X^u,X, X^{\rm fd} \}$ for every $i=1,\ldots,n$, and $X_k = X^u $ for some $k\in\{1,\ldots,n\}$. Moreover, in this case $\Pi_{X_1,\ldots,X_n;Y} = \Pi_{X_1,\ldots,X_n;Y^u}$.
\end{theorem}

\begin{proof} Let $X_1, \ldots, X_n$ be sequence classes as in the statement of the theorem. Given a Banach space $E$ and $(x_j)_{j=1}^\infty \in X(E)$, using that $X$ is finitely shrinking and finitely zero invariant we get, for any $m \in \mathbb{N}$, %como $X$ é finitamente zero relativo,
	\begin{align*}
	\|(x_j)_{j=1}^m\|_{X(E)}&=\|(x_j)_{j=1}^\infty-(\underbrace{0,\ldots,0}_{m \,{\rm times}},x_{m+1},x_{m+2},\ldots) \|_{X(E)}\\
	&\leq \|(x_j)_{j=1}^\infty\|_{X(E)}+\|(\underbrace{0,\ldots,0}_{m\, {\rm times}},x_{m+1},x_{m+2},\ldots) \|_{X(E)}\\
	&\leq \|(x_j)_{j=1}^\infty\|_{X(E)}+\|(x_j)_{j=1}^\infty\|_{X(E)}= 2 \|(x_j)_{j=1}^\infty\|_{X(E)}.
	\end{align*}
It follows that $(x_j)_{j=1}^\infty \in X^{\rm fd}(E)$, which proves that $X \subseteq X^{\rm fd}$. From Proposition \ref{propxfd}(b) we know that $X^{\rm fd}$ and $Y^{\rm fd}$ are finitely determined sequence classes, $\overline{c_{00}}^X = X^u$ is a closed subclass of $X^{\rm fd}$ and $\overline{c_{00}}^Y=Y^u$ is a closed subclass of $Y^{\rm fd}$. In particular, $X\stackrel{\rm fin}{=} X^{\rm fd}$ and $Y^u\stackrel{\rm fin}{=} Y^{\rm fd}$. We wish to apply Proposition \ref{cap3T01} to conclude that
\begin{equation}\label{pl3b}\Pi_{X_1,\ldots,X_n;Y^u}=\Pi_{^nX^{\rm fd};Y^{\rm fd}}. \end{equation}
It is enough to make a checklist of the assumptions of Proposition \ref{cap3T01} in this case: $X^{\rm fd}$ and $Y^{\rm fd}$ are finitely shrinking and finitely determined (Proposition \ref{propxfd}(b)); $Y^u\stackrel{\rm fin}{=} Y^{\rm fd}$ (see above); $X^u \subseteq X$ (obvious) and $X \subseteq X^{\rm fd}$ (see above); $X^u \stackrel{\rm fin}{=} X$ (obvious) and $X \stackrel{\rm fin}{=} X^{\rm fd}$ (see above); $X^u = (X^{\rm fd})^u$ (Proposition \ref{propxfd}(d)), so $X_k = (X^{\rm fd})^u$ for some $k$; $(Y^{\rm fd})^u = Y^u $ (Proposition \ref{propxfd}(d)). Hence, (\ref{pl3b}) holds. Given $A\in \Pi_{X_1,\ldots,X_n;Y}(E_1,\ldots ,E_n;F) $, using again that $Y^u\stackrel{\rm fin}{=} Y^{\rm fd}$ and $X_i \stackrel{fin}{=} X^{\rm fd}$, we have
	\begin{align*}
	\|(A(x_j^1,\ldots ,x_j^n))_{j=1}^{k}&\|_{Y^{\rm fd}(F)}=\|(A(x_j^1,\ldots ,x_j^n))_{j=1}^{k}\|_{Y^u(F)} = \|(A(x_j^1,\ldots ,x_j^n))_{j=1}^{k}\|_{Y(F)}\\
	&\leq \|A\|_{X_1,\ldots,X_n;Y}\cdot \prod_{i=1}^n \|(x_j^i)_{j=1}^k\|_{X_i(E_i)}= \|A\|_{X_1,\ldots,X_n;Y}\cdot \prod_{i=1}^n \|(x_j^i)_{j=1}^k\|_{X^{\rm fd}(E_i)}
	\end{align*}
for all $k\in\mathbb{N}$ and $x_j^i\in E_i$, $j=1,\ldots, k$, $i=1,\ldots, n$. As $X^{\rm fd}$ and $Y^{\rm fd}$ are finitely determined, therefore it follows that $A\in \Pi_{^nX^{\rm fd};Y^{\rm fd}}(E_1,\ldots ,E_n;F)  $, therefore  $\Pi_{X_1,\ldots,X_n;Y}\subseteq \mathcal{L}_{^nX^{\rm fd};Y^{\rm fd}}$. Combining this inclusion with the obvious inclusion $Y^u \subseteq Y$ and (\ref{pl3b}) we have
	$$ \Pi_{X_1,\ldots,X_n;Y^u}\subseteq \Pi_{X_1,\ldots,X_n;Y}\subseteq \Pi_{^nX^{\rm fd};Y^{\rm fd}}=\Pi_{X_1,\ldots,X_n;Y^u}.$$
%from which the result follows.%	portanto
Therefore, $\Pi_{X_1,\ldots,X_n;Y} %=\mathcal{L}_{X_1,\ldots,X_n;Y}
= \Pi_{^nX^{\rm fd};Y^{\rm fd}}$ is a symmetric ideal.
\end{proof}

\begin{example}\rm Let $1 \leq p < \infty$. Since $\ell_p^w$ and $\ell_p^{\rm mid}$ are finitely determined and finitely zero invariant, $(\ell_p^w)^u = \ell_p^u \neq \ell_p^w$ and $\ell_p^{\rm mid} \neq (\ell_p^{\rm mid})^u$,   $$\Pi_{\ell_p^{\theta_1},\ldots, \ell_p^{\theta_n};Y } \mbox{ and~} \Pi_{X_1,\ldots, X_n;Y } $$
 are non trivial sequential symmetric ideals for every $n \geq 2$, all $\theta_1, \ldots, \theta_n \in \{u,w\}$ with $\theta_i = u$ and $\theta_j = w$ for some $i$ and $j$, all $X_1, \ldots, X_n \in \{\ell_p^{\rm mid}, (\ell_p^{\rm mid})^u\}$ with $X_i = \ell_p^{\rm mid}$ and $X_j = (\ell_p^{\rm mid})^u$ for some $i$ and $j$, and every finitely shrinking and finitely zero invariant sequence class $Y \neq \ell_\infty(\cdot)$.

 In particular, $\Pi_{\ell_2^w, \ell_2^u; {\rm Rad}}$ is a symmetric ideal of bilinear operators, information which cannot be recovered using the previous section
 \end{example}

\begin{example}\rm  Since RAD is finitely shrinking and finitely zero invariant, for every $n \geq 2$, the ideal $\Pi_{X_1, \ldots, X_n;Y}$ is symmetric whenever $X_i \in \{ {\rm RAD}, {\rm Rad}\}$ with $X_i = {\rm RAD}$ and $X_j = {\rm Rad}$ for some $i$ and $j$, and $Y \neq \ell_\infty(\cdot)$ is a finitely shrinking and finitely zero invariant sequence class.
\end{example}

\begin{remark}\rm The reader might be wondering if $X \subseteq X^{\rm fd}$ for every sequence class $X$. All we know is that $X$ is a closed subclass of $X^{\rm fd}$  if $\|(x_j)_{j=1}^\infty\|_{X(E)} = \sup\limits_k  \|(x_j)_{j=1}^k\|_{X(E)}$ for every $(x_j)_{j=1}^\infty \in X(E)$ and any Banach space $E$. Note that this condition is weaker than being finitely determined.
\end{remark}

\section{Further applications}

In this section we show that the procedures $X \mapsto X^u$ and $X \mapsto X^{\rm fd}$ and the results proved in the previous sections can be useful beyond the scope of symmetric ideals of multilinear operators.

As mentioned in the Introduction, the class that originated this line of research is the ideal $\Pi_p$ of absoluely $p$-summing linear operators, defined as $\Pi_p = \Pi_{\ell_p^w;\ell_p(\cdot)}$. It is well known that it is enough to check the transformation of sequences in $\ell_p^u$, that is, $\Pi_p = \Pi_{\ell_p^w;\ell_p(\cdot)} = \Pi_{\ell_p^u;\ell_p(\cdot)} $. Let us see that this coincidence is a particular case of much more general ones.

\begin{proposition} Let $X$ be a finitely determined and finitely shrinking sequence class.\\
{\rm (a)} $\Pi_{X;Y} \stackrel{1}{=} \Pi_{X^u;Y}$ for every finitely determined sequence class $Y$.\\
{\rm (b)} $\Pi_{X;Y} \stackrel{1}{=} \Pi_{X^u;Y^u}$ for every finitely determined and finitely shrinking sequence class $Y$.\\
{\rm (c)} $\Pi_{X^u;Y^u} = \Pi_{X;Y^{\rm fd}} \stackrel{1}{=} \Pi_{X^u; Y^{\rm fd}}$ for every finitely shrinking and finitely zero invariant sequence class $Y$.
\end{proposition}

\begin{proof} (a) Since $X^u$ is a closed subclass of $X$, the desired equality was obtained in the proof of Proposition \ref{cap3T01}.\\
(b) Let $T\in \Pi_{X^u;Y}(E;F)$ and $ (x_j)_{j=1}^\infty \in X^u(E) $ be given. For every $n\in\mathbb{N}$, we have $(x_j)_{j=n}^\infty\in X(E)$ and $(T(x_j))_{j=n}^\infty\in Y(F)$ because $X$ and $Y$ are finitely shrinking. Thus,
	$$ \|(T(x_j))_{j=n}^\infty\|_{Y(F)} \leq \|T\|_{X^u;Y} \|(x_j)_{j=n}^\infty\|_{X(E)} \stackrel{n \to \infty}{\longrightarrow} 0,$$
showing that $T \in  \Pi_{X^u;Y^u}(E;F)$. It follows that $\Pi_{X^u;Y} \stackrel{1}{\hookrightarrow} \Pi_{X^u;Y^u}$. %Now, given $T\in \Pi_{X^u;Y^u}(E;F)$ and $(x_j)_{j=1}^\infty \in X^u(E) $, we have  temos
%	$$ \|(T(x_j))_{j=1}^\infty\|_{Y(F)} \leq \|T\|_{X^u;Y^u} \|(x_j)_{j=1}^\infty\|_{X(E)},$$
%	portanto
The inclusion $\Pi_{X^u;Y^u} \stackrel{1}{\hookrightarrow} \Pi_{X^u;Y}$ is straightforward, so $\Pi_{X^u;Y^u} \stackrel{1}{=} \Pi_{X^u;Y}$. The result follows from (a).\\ % Para concluir, basta aplicar o item (i).
(c) By Proposition \ref{propxfd}(b), $Y^{\rm fd}$ is finitely determined. From (a) it follows that $\Pi_{X; Y^{\rm fd}} \stackrel{1}{=} \Pi_{X^u; Y^{\rm fd}}$. We know that $Y^{\rm fd}$ is finitely shrinking (Proposition \ref{propxfd}(b)) and $(Y^{\rm fd})^u = Y^u$ (Proposition \ref{propxfd}(d)), so $\Pi_{X;Y^{\rm fd}} = \Pi_{X^u;Y^u}$ by Proposition \ref{cap3T01}.%  Para concluir o item (iii) basta aplicar o Teorema 3.34.
\end{proof}

\begin{example}\rm Let $1 \leq p < \infty$.  Taking $X = \ell_p^w$ and $Y = \ell_p(\cdot)$ we recover the known coincidence $ \Pi_{\ell_p^w;\ell_p(\cdot)} = \Pi_{\ell_p^u;\ell_p(\cdot)} $. And taking $X = \ell_p^w$ and $Y = {\rm RAD}$ we get \begin{equation}\Pi_{\ell_p^u; {\rm Rad}} = \Pi_{\ell_p^w; {\rm RAD}}= \Pi_{\ell_p^u; {\rm RAD}}. \label{t6bq}
\end{equation}
Since ${\rm Rad}(\mathbb{K}) = {\rm RAD}(\mathbb{K}) = \ell_2$, the class of operators in (\ref{t6bq}) is of interest (in the sense that it is an operator ideal) only for $1 \leq p \leq 2$. The information in (\ref{t6bq}) generalizes the equality $\Pi_{\ell_2^u; {\rm Rad}} = \Pi_{\ell_2^w; {\rm RAD}}$ obtained in \cite{botelho+campos} for the ideal of almost summing linear operators (see \cite{diestel+jarchow+tonge}).
\end{example}

 For $1 \leq p < \infty$, we are interested in sequence classes $X$ having $\ell_p$ as scalar component, that is, $X(\mathbb{K}) = \ell_p$. Contemplating the following chain of the most usual of such sequence classes
$$\ell_p\langle\cdot \rangle \subseteq \ell_p(\cdot) \subseteq \ell_p^u \cap \ell_p^{\rm mid} \subseteq \ell_p^u \cup \ell_p^{\rm mid} \subseteq \ell_p^w, $$
it is inevitable to wonder if every sequence class having $\ell_p$ as scalar component lies between $\ell_p \langle \cdot \rangle$ and $\ell_p^w$. On the one hand, it is easy to see that $\ell_p^w$ is the greatest sequence class having $\ell_p$ as scalar component. More precisely, if $X$ is a sequence class with $X(\mathbb{K}) = \ell_p$, then $X \stackrel{1}{\hookrightarrow} \ell_p^w$. This fact follows immediately from the linear stability of $X$ applied to linear functionals. On the other hand, we shall prove next that $\ell_p \langle \cdot \rangle$ is minimal amongst the sequence classes having $\ell_p$ as scalar component and enjoying some properties shared by all usual classes. We shall do that using the procedures $X \mapsto X^u$ and $X \mapsto X^{\rm fd}$ from the previous sections, as well as the procedure $X \mapsto X^{\rm dual}$ from \cite{botelho+campos+2}, which we describe next.

According to \cite{botelho+campos+2}, a sequence class $X$ is said to be:\\
$\bullet$ {\it Finitely injective} if $\|(x_j)_{j=1}^n\|_{X(E)} \leq \|(i(x_j))_{j=1}^n\|_{X(F)} $ for every metric injection $i\colon E\longrightarrow F$ between Banach spaces and all $n\in\mathbb{N}$, $x_1,\ldots,x_n \in E$.\\
$\bullet$ {\it Spherically complete} if $(\alpha_j x_j)_{j=1}^\infty\in X(E)$ and $\|(\alpha_j  x_j)_{j=1}^\infty\|_{X(E)}=\|( x_j)_{j=1}^\infty\|_{XE}$ whenever $( x_j)_{j=1}^\infty\in X(E)$ and $(\alpha_j)_{j=1}^\infty\in \mathbb{K}^{\mathbb{N}}$, with $|\alpha_j|=1$ for every $j\in\mathbb{N}$.

Given a sequence class $X$ and a Banach space $E$, define
\begin{equation} \label{po9n}X^{\rm dual}(E) := \left\{(x_j)_{j=1}^\infty \in E^{\mathbb{N}} : \sum_{j=1}^\infty x_j^*(x_j) \mbox{ converges for every }(x_j^*)_{j=1}^\infty \in X(E^*)\right\},\end{equation}
$$\|(x_j)_{j=1}^\infty\|_{X^{\rm dual}(E)} := \sup_{(x_j^*)_{j=1}^\infty \in B_{X(E^*)}}\sum_{j=1}^\infty |x_j^*(x_j)|. $$
If $X$ is spherically complete, then $X^{\rm dual}$ is a (linearly stable)  finitely determined and spherically complete sequence class \cite[Proposition 2.5]{botelho+campos+2}.

\begin{proposition} Let $1 \leq p < \infty$ and let $X$ be a finitely injective, finitely zero invariant, finitely shrinking  and spherically complete sequence class. If $X(\mathbb{K}) = \ell_p$ and $X \stackrel{1}{\hookrightarrow} \ell_p \langle \cdot \rangle$, then $X \stackrel{1}{=} \ell_p \langle \cdot \rangle$.
\end{proposition}

\begin{proof} We already know that $X^u$ is a (linearly stable) sequence class and  $X^u=\overline{c_{00}}^X$ (Proposition \ref{propxu}). Using that $X$ is finitely shrinking and spherically complete, it follows easily that $X^u$ is spherically complete, so we can consider the sequence class $(X^u)^{\text{dual}}$. %Vejamos que a hipótese de $X$ ser finitamente contrátil garante que  $X^u$ também seja esfericamente completa. De fato, sejam $(x_j)_{j=1}^\infty \in X^u(E)$ e $(\lambda_j)_{j=1}^\infty \in \mathbb{K}^\mathbb{N}$, com $|\lambda_j|=1$ para todo $j\in\mathbb{N}$. Como $X$ é finitamente contrátil e esfericamente completa então;
	%$$(x_j)_{j=n}^\infty \in X(E),~(\lambda_j)_{j=n}^\infty \in \mathbb{K}^\mathbb{N} \Rightarrow  (\lambda_j x_j)_{j=n}^\infty \in X(E) ~\text{e}~ \|(\lambda_j x_j)_{j=n}^\infty\|_{X(E)}= \|( x_j)_{j=n}^\infty\|_{X(E)}$$
%	para todo $n\in\mathbb{N}$, portanto $(\lambda_j x_j)_{j=1}^\infty \in X^u(E)$. Além disso,
It is clear that $X^u(\mathbb{K})=\ell_p$. Denoting $(\ell_p)_* = \ell_{p^*}$ if $1 < p < \infty$ and $(\ell_1)_* = c_0$, from (\ref{po9n}) we get $(X^u)^{\text{dual}}(\mathbb{K}) = (\ell_{p})_*$.
	%\begin{align*}
%	(X^u)^{\text{dual}}(\mathbb{K})&=\left\{(\beta_j)_{j=1}^\infty \in \mathbb{K}^\mathbb{N}:\sum_{j=1}^{\infty} \alpha_j \beta_j ~\text{converge~para~todo}~(\alpha_j)_{j=1}^\infty \in X^u(\mathbb{K})=\ell_{p}\right\}\\
%	&=(\ell_{p})_*.
%	\end{align*}
Let $E$ be a Banach space. The linear stability of $(X^u)^{\text{dual}}$ gives that $(x^*(x_j))_{j=1}^\infty \in (X^u)^{\text{dual}}(\mathbb{K})= \ell_{p^*}$ for all $x^* \in E^*$ and $(x_j)_{j=1}^\infty \in (X^u)^{\text{dual}}(E)$. Moreover, % obtemos , pois $X$ linearmente estável implica $X^u$ linearmente estável que por sua vez implica $(X^u)^{\text{dual}}$ linearmente estável,
%portanto
$ (X^u)^{\text{dual}}(E) \stackrel{1}{\hookrightarrow} \ell_{p^*}^w(E)$; in particular, $B_{ (X^u)^{\text{dual}}(E)} \subseteq B_{\ell_{p^*}^w(E)} .$
	From Proposition \ref{propxfd}(b) and (d) we know that $X^{\rm fd}$ is finitely determined and that $X^u=\overline{c_{00}}^X$ is a closed subclass of $X^{\rm fd}$. And from the proof of Proposition \ref{propxfd}(b) (see (\ref{d4eb})) we have that % , mais ainda, como $X$ é finitamente contrátil as normas
$\|\cdot\|_{X(E)}= \|\cdot\|_{X^{\rm fd}(E)}$ on $c_{00}(E)$, hence  $X^u=\overline{c_{00}}^X=\overline{c_{00}}^{X^{\rm fd}}$. In the terminology of \cite{botelho+campos+2} this implies that $X^u$ is a finitely dominated sequence class. We have shown that $X^u$ fulfills all the assumptions of \cite[Theorem 2.10]{botelho+campos+2}, so the spaces $(X^u)^{\rm dual}(E^*)$ and $[X^u(E)]^*$ are isometrically isomorphic by means of the correspondence
$$ x^* = (x_j^*)_{j=1}^\infty \in (X^u)^{\rm dual}(E^*) \mapsto \varphi_{x^*} \in [X^u(E)]^*, \varphi_{x^*}((x_j)_{j=1}^\infty) = \sum_{j=1}^\infty x_j^*(x_j),  $$
for every $(x_j)_{j=1}^\infty \in X^u(E)$.
%Temos então, todas as condições para usar o Teorema \ref{capPT2}, isto é, $(X^u)^{\text{dual}}(E')\cong (X^u(E))'$. Daí, dado
Applying the Hahn-Banach Theorem, for any $(x_j)_{j=1}^\infty \in X^u(E)$ we have
	\begin{align*}
	\|(x_j)_{j=1}^\infty\|_{X(E)}&=\|(x_j)_{j=1}^\infty\|_{X^u(E)}=\sup_{\varphi\in B_{[X^u(E)]^*}} \left| \varphi((x_j)_{j=1}^\infty ) \right|\\
	&=\sup_{(x_j^*)_{j=1}^\infty \in B_{(X^u)^{\text{dual}}(E^*)}} \left| \sum_{j=1}^{\infty} x_j^*(x_j) \right|\leq\sup_{(x_j^*)_{j=1}^\infty \in B_{\ell_{p^*}^w(E)}} \left| \sum_{j=1}^{\infty} x_j^*(x_j) \right|= \|(x_j)_{j=1}^\infty\|_{\ell_p\langle E\rangle},
	\end{align*}
By assumption we have $X \stackrel{1}{\hookrightarrow} \ell_p\langle \cdot \rangle $, so the inequality above gives that $X^u$ is a closed subclass of $\ell_p\langle \cdot\rangle$. It follows that
	$$\overline{c_{00}}^{\ell_p\langle \cdot \rangle} (E) \stackrel{1}{\hookrightarrow} X^u(E) \stackrel{1}{\hookrightarrow} X(E) \stackrel{1}{\hookrightarrow} \ell_p\langle E \rangle \stackrel{1}{=} \overline{c_{00}}^{\ell_p\langle \cdot \rangle} (E),$$
from which we get $X \stackrel{1}{=} \ell_p \langle \cdot \rangle$.
\end{proof}

\bigskip

\noindent Geraldo Botelho~~~~~~~~~~~~~~~~~~~~~~~~~~~~~~~~~~~~~~Ariel S. Santiago\\
Faculdade de Matem\'atica~~~~~~~~~~~~~~~~~~~~~~~~~~Departamento de Matemática\\
Universidade Federal de Uberl\^andia~~~~~~~~~~~~~Universidade Federal de Minas Gerais\\
38.400-902 -- Uberl\^andia -- Brazil~~~~~~~~~~~~~~~~~31.270-901 -- Belo Horizonte -- Brazil\\
e-mail: botelho@ufu.br~~~~~~~~~~~~~~~~~~~~~~~~~\,~~~~~e-mail: arielsantossantiago@hotmail.com

%\medskip
%
%\noindent Ariel S. Santiago\\
%Departamento de Matemática\\
%Universidade Federal de Minas Gerais\\
%???? -- Belo Horizonte -- Brazil\\
%e-mail: ?????

%\bigskip
%\bigskip
%
%$F$ preserva multimorfismos de Riesz se as extensões de Aron-Berner de $F$-valued MR são MR
%
%\bigskip
%
%Teorema. Se $F_n$ preserva MR para todo $n$ então $(\oplus_n F_n)_p$ e $(\oplus_n F_n)_0$ preservam MR

\end{document}